\documentclass[brochure,english,12pt]{bourbaki}
\usepackage[matrix,arrow]{xy}
\usepackage{amssymb,amsfonts,amsmath,footnote}
\usepackage{enumerate}
\usepackage{cancel} 
\usepackage{mathrsfs} 
\usepackage{colonequals}
\usepackage[T1]{fontenc} 

\usepackage[OT2,T1]{fontenc}
\DeclareSymbolFont{cyrletters}{OT2}{wncyr}{m}{n}
\DeclareMathSymbol{\Sha}{\mathalpha}{cyrletters}{"58}

\usepackage{color}

\newcommand{\defin}[1]{\textbf{\textsf{#1}}} 
\newcommand{\leftquot}[2]{\left. \raisebox{-0.5ex}{$#1$} \middle\backslash \raisebox{0.5ex}{$#2$} \right.}
\newcommand{\rightquot}[2]{\left. \raisebox{0.5ex}{$#1$} \middle/ \raisebox{-0.5ex}{$#2$} \right.}

\newcommand{\C}{\mathbb{C}}
\newcommand{\F}{\mathbb{F}}

\newcommand{\PP}{\mathbb{P}}
\newcommand{\Q}{\mathbb{Q}}
\newcommand{\R}{\mathbb{R}}
\newcommand{\Z}{\mathbb{Z}}
\newcommand{\Qbar}{{\overline{\Q}}}
\newcommand{\Zhat}{{\widehat{\Z}}}

\newcommand{\kbar}{{\overline{k}}}

\newcommand{\Adeles}{\mathbf{A}}

\newcommand{\calE}{\mathcal{E}}
\newcommand{\calF}{\mathcal{F}}

\newcommand{\calS}{\mathcal{S}}

\newcommand{\EE}{\mathscr{E}}
\newcommand{\FF}{\mathscr{F}}

\DeclareMathOperator{\Average}{Average}

\DeclareMathOperator{\Char}{char}

\DeclareMathOperator{\corank}{corank}

\DeclareMathOperator{\Gal}{Gal}

\DeclareMathOperator{\image}{image}

\DeclareMathOperator{\Jac}{Jac}

\DeclareMathOperator{\ls}{ls}

\newcommand{\nl}{\textup{\:\cancel{linear}}}

\DeclareMathOperator{\ord}{ord}

\DeclareMathOperator{\Prob}{Prob}

\DeclareMathOperator{\Sel}{Sel}
\DeclareMathOperator{\Spec}{Spec}

\DeclareMathOperator{\Tam}{Tam}

\newcommand{\an}{{\operatorname{an}}}

\newcommand{\GL}{\operatorname{GL}}
\newcommand{\HH}{{\operatorname{H}}}

\newcommand{\PGL}{\operatorname{PGL}}

\newcommand{\SL}{\operatorname{SL}}

\newcommand{\directsum}{\oplus} 
\newcommand{\injects}{\hookrightarrow}
\newcommand{\intersect}{\cap} 
\newcommand{\isom}{\simeq}

\newcommand{\tensor}{\otimes} 
\newcommand{\To}{\longrightarrow}

\usepackage[
	pdfauthor={Bjorn Poonen}, 
]{hyperref}
\usepackage[alphabetic,lite]{amsrefs} 

\date{Janvier 2012 / revised June 13, 2015}
\bbkannee{64\`eme ann\'ee, 2011-2012}
\bbknumero{1049}
\title{Average rank of elliptic curves}
\subtitle{after Manjul Bhargava and Arul Shankar}
\author{Bjorn POONEN}
\address{Department of Mathematics \\ Massachusetts Institute of Technology \\ 77 Massachusetts Avenue \\ Cambridge, MA 02139-4307, U.S.A.}
\email{poonen@math.mit.edu}
\urladdr{http://math.mit.edu/~poonen/}

\begin{document}

\maketitle

\section{Introduction}
\subsection{Elliptic curves}

An \defin{elliptic curve} $E$ over $\Q$ is the projective closure of
a curve $y^2=x^3+Ax+B$ for some fixed $A,B \in \Q$
satisfying $4A^3+27B^2 \ne 0$ 
(the inequality is the condition for the curve to be smooth).
Such curves are interesting because
\begin{enumerate}
\item they are the simplest 
algebraic varieties whose rational points are not completely understood, and
\item they are the simplest examples of projective algebraic groups
of positive dimension.
\end{enumerate}
The abelian group $E(\Q)$ of rational points on $E$ 
is finitely generated~\cite{Mordell1922}.
Hence $E(\Q) \isom \Z^r \directsum T$
for some nonnegative integer $r$ (the \defin{rank})
and some finite abelian group $T$ (the \defin{torsion subgroup}).
The torsion subgroup is well understood, 
thanks to B.~Mazur~\cite{Mazur1977},
but the rank remains a mystery.
Already in 1901, H.~Poincar\'e~\cite{Poincare1901}*{p.~173} 
asked what is the range of possibilities for the minimum number 
of generators of $E(\Q)$,
but it is not known even whether $r$ is bounded.
There are algorithms that compute $r$ successfully in practice,
given integers $A$ and $B$ of moderate size,
but to know that the algorithms terminate in general,
it seems that one needs a conjecture: 
either the finiteness of the Shafarevich--Tate group $\Sha$
(or of its $p$-primary part for some prime $p$),
or the Birch and Swinnerton-Dyer conjecture 
that $r$ equals the \defin{analytic rank} 
$r_{\an} \colonequals \ord_{s=1} L(E,s)$~\cite{Birch-Swinnerton-Dyer1965}.

The main results of Bhargava and Shankar (Section~\ref{S:main results}) 
concern the average value of $r$ 
as $E$ ranges over all elliptic curves over $\Q$.

\subsection{Selmer groups}

There is essentially only one known proof that $E(\Q)$ is finitely generated.
The hardest step involves proving the finiteness of $E(\Q)/nE(\Q)$
for some $n \ge 2$.
This is done by embedding $E(\Q)/nE(\Q)$ into 
the \defin{$n$-Selmer group} $\Sel_n(E)$,
which we now define.

For each prime $p$, let $\Q_p$ be the field of $p$-adic numbers;
also define $\Q_\infty \colonequals  \R$.
Let $\Qbar$ be an algebraic closure of $\Q$.
We write $\HH^1(\Q,E)$, for example, 
to denote the profinite group cohomology $\HH^1(\Gal(\Qbar/\Q),E(\Qbar))$.

Fix $n \ge 2$.
For any abelian group or group scheme $G$, let $G[n]$ be the kernel
of multiplication-by-$n$ on $G$.
Taking cohomology of
\[
	0 \To E[n] \To E \stackrel{n}\To E \to 0
\]
over $\Q$ and $\Q_p$ leads to the exact rows in the commutative diagram
\begin{equation}
\label{E:descent}
\begin{split}
\xymatrix{
0 \ar[r] & \dfrac{E(\Q)}{nE(\Q)} \ar[r] \ar[d] & \HH^1(\Q,E[n]) \ar[r] \ar[d]^-{\beta} & \HH^1(\Q,E)[n] \ar[r] \ar[d] & 0 \\
0 \ar[r] & \displaystyle\prod_{p \le \infty} \dfrac{E(\Q_p)}{nE(\Q_p)} \ar[r]^-{\alpha} & \displaystyle\prod_{p \le \infty} \HH^1(\Q_p,E[n]) \ar[r] & \displaystyle\prod_{p \le \infty} \HH^1(\Q_p,E)[n] \ar[r] & 0 \\
}
\end{split}
\end{equation}
The group $\HH^1(\Q,E[n])$ turns out to be infinite,
and it is difficult to determine which of its elements are 
in the image of $E(\Q)/nE(\Q)$.
But because arithmetic over $\Q_p$ is easier than arithmetic over $\Q$,
one can determine which elements are \emph{locally} in the image.
With this in mind, define
\[
	\Sel_n(E) \colonequals \{ x \in \HH^1(\Q,E[n]) : \beta(x) \in \image(\alpha) \}
\]
Diagram~\eqref{E:descent} shows that 
the subgroup $\Sel_n(E) \subseteq \HH^1(\Q,E[n])$ is an upper
bound for the image of $E(\Q)/nE(\Q)$.
In fact, if we define also the \defin{Shafarevich--Tate group}
\[
	\Sha = \Sha(E) \colonequals \ker\left(\HH^1(\Q,E) \to \prod_{p \le \infty} \HH^1(\Q_p,E)\right),
\]
then diagram~\eqref{E:descent} yields an exact sequence
\begin{equation}
\label{E:Selmer-Sha}
	0 \To \frac{E(\Q)}{nE(\Q)} \To \Sel_n(E) \To \Sha[n] \To 0.
\end{equation}
Moreover, it turns out that $\Sel_n(E)$ is finite and computable.

\subsection{Averaging over all elliptic curves}

The average of an infinite sequence of real numbers $a_1,a_2,\ldots$
is defined as \hbox{$\lim_{n \to \infty} (a_1+\cdots+a_n)/n$},
if the limit exists.
This may depend on the ordering of the terms.
Hence, to define the average rank of elliptic curves,
we should first decide how to order them.

Tables such as \cites{AntwerpIV,Cremona1997,CremonaTables,SteinTables}
order elliptic curves by their conductor $N$.
But it is not known even how many elliptic curves have
conductor $<X$ asymptotically as $X \to \infty$,
so we cannot hope to prove anything nontrivial about averages 
for this ordering.
Ordering by minimal discriminant runs into the same difficulty.

Therefore we order by height, which we now define.
Elliptic curves $y^2=x^3+Ax+B$ and $y^2=x^3+A'x+B'$ over $\Q$
are isomorphic if and only if there exists $q \in \Q^\times$
such that $(A',B')=(q^4 A,q^6 B)$.
Therefore each isomorphism class contains a unique representative
$E_{AB}$ with $(A,B) \in \Z^2$ \defin{minimal}
in the sense that there is no prime $p$ with $p^4 | A$ and $p^6 | B$.
Let $\EE$ be the set of all such $E_{AB}$.
Define the (na\"{\i}ve) \defin{height} 
$H(E_{AB})=H(A,B)\colonequals \max\{|4A^3|,27B^2\}$.
(Other authors replace $4$ and $27$ by other positive constants;
it is only the \emph{exponents} that matter in the proofs.)
For $X \in \R$, define $\EE_{<X}\colonequals \{E \in \EE: H(E) < X\}$.
For any $\phi \colon \EE \to \R$, define its \defin{average} by
\[
	\Average(\phi)\colonequals \lim_{X \to \infty} \frac{\sum_{E \in \EE_{<X}} \phi(E)}{\sum_{E \in \EE_{<X}} 1},
\]
if the limit exists.
Define $\overline{\Average}(\phi)$ and $\underline{\Average}(\phi)$ similarly, 
but using $\limsup$ or $\liminf$, respectively.

We may speak also of the \defin{probability} or \defin{density}
of the set of elliptic curves satisfying a given property.
Namely, the property $P$ can be identified with its characteristic function 
$\chi_P \colon \EE \to \{0,1\}$;
then define $\Prob(P)=\Average(\chi_P)$.
Similarly define $\overline{\Prob}(P)$ and $\underline{\Prob}(P)$.

\begin{exem}
\label{E:torsion}
B.~Mazur's theorem~\cite{Mazur1977}
bounds the possibilities for the torsion subgroup $T$.
The Hilbert irreducibility theorem shows that each nonzero possibility
for $T$ occurs rarely.
Together, they show that $\Prob(T \ne 0)$ is $0$.
\end{exem}

\subsection{Main results of Bhargava and Shankar}
\label{S:main results}

\begin{theo}[\cite{Bhargava-Shankar-2selmer}*{Theorem~1.1}]
\label{T:Sel2}
$\Average(\#\Sel_2) = 3$.
\end{theo}

\noindent
If one averages not over all of $\EE$, but over a subset
defined by finitely many congruence conditions on $A$ and $B$
(e.g., $A \equiv 5 \pmod{7}$ and $B \equiv 3 \pmod{4}$),
then the average is still $3$~\cite{Bhargava-Shankar-2selmer}*{Theorem~1.3}.
This is interesting, 
given that one of the successful techniques for constructing
elliptic curves of moderately large rank has been to restrict attention
to congruence classes so as to maximize $\#E(\F_p)$ for the first few
primes $p$ \cite{Mestre1982}.

A similar argument leads to

\begin{theo}[\cite{Bhargava-Shankar-3selmer}*{Theorem~1}]
\label{T:Sel3}
$\overline{\Average}(\#\Sel_3) \le 4$.\footnote{After the original version of this survey was written, Bhargava and Shankar proved the stronger statement $\Average(\#\Sel_3)=4$.}
\end{theo}

\noindent Again one can obtain the same bound for elliptic curves
satisfying finitely many congruence conditions.
One can even impose congruence conditions
at \emph{infinitely} many primes as long as one can show that
the conditions at large primes together are sieving out a negligible subset.

It is still not known whether $\Average(r)$ exists, 
but Theorems \ref{T:Sel2} and~\ref{T:Sel3} 
yield upper bounds on $\overline{\Average}(r)$:

\begin{coro}[\cite{Bhargava-Shankar-3selmer}*{Corollary~2}]
\label{C:rank}
$\overline{\Average}(r) \le 7/6$.
\end{coro}

\begin{proof}
Let $s = \dim \Sel_3$.
The injection $E(\Q)/3E(\Q) \injects \Sel_3(E)$ yields $r \le s$.
Combining this with $6s-3 \le 3^s$ bounds $r$ in terms of $\#\Sel_3$;
then apply $\overline{\Average}$ and use Theorem~\ref{T:Sel3}.
(Why $6s-3$?
Since $3^s$ is a convex function, it suffices to connect
the points $(s,3^s)$ for $s=0,1,\ldots$ in order by line
segments, and to take the equation of the line segment that
crosses the horizontal line $y=4$.)
\end{proof}

Further consequences of Theorem~\ref{T:Sel3}
make use of results of Dokchitser--Dokchitser and Skinner--Urban,
whose context can be best understood if we introduce a few more quantities.
Taking the direct limit of~\eqref{E:Selmer-Sha} as $n$ ranges through
powers of a prime $p$ yields 
the \defin{$p^\infty$-Selmer group} $\Sel_{p^\infty}(E)$ 
fitting in an exact sequence
\begin{equation}
\label{E:p^infty Selmer}
	0 \To E(\Q) \tensor \frac{\Q_p}{\Z_p} \To \Sel_{p^\infty}(E) 
	\To \Sha[p^\infty] \To 0.
\end{equation}
Each term in~\eqref{E:p^infty Selmer} has the form 
\hbox{$(\Q_p/\Z_p)^c \directsum \textup{(finite)}$} 
for some nonnegative integer $c$ called the \defin{corank}.
Let $r_{p^\infty} \colonequals \corank \Sel_{p^\infty}(E)$.
Let $s_p'\colonequals \dim \Sel_p(E) - \dim E[p](\Q)$.
If $\Sha'$ is the quotient of $\Sha$ by its maximal divisible subgroup,
then $s_p'-r_{p^\infty}=\dim \Sha'[p]$, which is even
since $\Sha'[p^\infty]$ is a finite group 
with a nondegenerate alternating pairing~\cite{Cassels1962-IV}.
By~\eqref{E:p^infty Selmer}, $r_{p^\infty}-r = \corank \Sha[p^\infty]$, 
which is $0$ if and only if $\Sha[p^\infty]$ is finite.
To summarize, 
\begin{equation}
\label{E:ranks}
	s_p' \; \equiv \; r_{p^\infty} \; \stackrel{\textup{ST}}= \; r \; \stackrel{\textup{BSD}}= \; r_{\an},
\end{equation}
where the congruence is modulo $2$, and the equalities labeled with
the initials of Shafarevich--Tate and Birch--Swinnerton-Dyer are conjectural.
Also,
\begin{equation}
\label{E:inequalities}
	\dim \Sel_p \; \ge \; s_p' \; \ge \; r_{p^\infty} \; \ge \; r.
\end{equation}

In the direction of the conjectural equality 
$r_{p^\infty}=r_{\an}$,
we have two recent theorems:

\begin{theo}[\cite{Dokchitser-Dokchitser2010}*{Theorem~1.4}]
\label{T:Dokchitser}
For every elliptic curve $E$ over $\Q$, 
we have $r_{p^\infty} \equiv r_{\an} \pmod{2}$.
\end{theo}

The \defin{root number} $w \in \{\pm 1\}$ of an elliptic curve $E$ over $\Q$
may be defined as the sign of the functional equation 
for the $L$-function $L(E,s)$, so the conclusion of
Theorem~\ref{T:Dokchitser} may also be written $(-1)^{s_p'}=w$.

\begin{theo}[\cite{Skinner-Urban2014}*{Theorem~2(b)}]
\label{T:Skinner-Urban}
For any odd prime $p$
and elliptic curve $E$ over $\Q$ satisfying mild technical hypotheses,
if $r_{p^\infty} = 0$, then $r_{\an}=0$.
\end{theo}

Combining Theorems \ref{T:Sel3} (with congruence conditions), 
\ref{T:Dokchitser}, and~\ref{T:Skinner-Urban}
leads to 

\begin{theo}[\cite{Bhargava-Shankar-3selmer}*{\S 4.1,4.2}] 
\label{T:rank 0 and 1}
\hfill
\begin{enumerate}[\upshape (a)]
\item \label{I:rank 0}
$\underline{\Prob}(\dim \Sel_3 = s_3' = r_{3^\infty} = r = r_{\an} = 0)$ 
is positive.
\item \label{I:rank 1}
$\underline{\Prob}(\dim \Sel_3 = s_3' = r_{3^\infty} = 1)$
is positive.
\end{enumerate}
\end{theo}

\begin{proof}[Sketch of proof]
Wong~\cite{Wong2001}*{\S 9} constructed a positive-density subset
$\FF \subset \EE$ such that whenever $E \in \FF$, its $(-1)$-twist
is in $\FF$ and has the opposite root number.
By Example~\ref{E:torsion}, 
one can assume that $E(\Q)[3]=0$ for every $E \in \FF$.
Then, by the sentence after Theorem~\ref{T:Dokchitser}, 
the parity of $\dim \Sel_3 = s_3'$ for $E \in \FF$ is equidistributed.
Moreover, $\FF$ can be chosen so that Theorem~\ref{T:Skinner-Urban}
applies to every $E \in \FF$,
and so that the conclusion of Theorem~\ref{T:Sel3} holds for $\FF$.
For large $X$, let $p_0$ be the proportion of curves in $\FF$
with $\dim \Sel_3 = 0$; define $p_1$ similarly.
Our bound on the average of $\#\Sel_3$ yields
\[
	p_0 \cdot 1 
	+ p_1 \cdot 3
	+ \left( \frac12-p_0 \right) \cdot 9
	+ \left( \frac12-p_1 \right) \cdot 27
	\;\le\; 4 + o(1).
\]
This, with $p_0,p_1 \le 1/2 + o(1)$,
implies $p_0 \ge 1/4 - o(1)$ and $p_1 \ge 5/12 - o(1)$.
This proves the bounds for $\dim \Sel_3  = s_3'$.
For $E \in \FF$, if $s_3'=0$, then $r_{3^\infty}=r=r_{\an}=0$ too
by~\eqref{E:inequalities} and Theorem~\ref{T:Skinner-Urban}.
If $s_3'=1$, then $r_{3^\infty}=1$ too 
since $s_3' - r_{3^\infty} \in 2\Z_{\ge 0}$.
\end{proof}

Theorem~\ref{T:rank 0 and 1}\eqref{I:rank 0} 
implies in particular that a positive proportion of elliptic curves
over $\Q$ have both rank $0$ and analytic rank $0$
and hence satisfy the Birch and Swinnerton-Dyer conjecture that $r=r_{\an}$.
Theorem~\ref{T:rank 0 and 1}\eqref{I:rank 1} 
implies a conditional statement for rank~$1$:

\begin{coro}[\cite{Bhargava-Shankar-3selmer}*{\S 4.1}] 
\label{C:rank 1}
If $\Sha(E)$ (or at least its $3$-primary part) is finite 
for every elliptic curve $E$ over $\Q$, 
then $\underline{\Prob}(r=1)$ is positive.
\end{coro}

\section{Previous work of other authors}
\label{S:previous work}

This section exists only to put the theorems above in context.
Readers impatient to understand the proof of Theorem~\ref{T:Sel2} 
may jump to Section~\ref{S:background}.

\subsection{Average analytic rank}

Using analogues of Weil's ``explicit formula'', many authors have
given conditional bounds on the average \emph{analytic} rank,
both for the family of quadratic twists of a fixed elliptic curve over $\Q$,
and for the family $\EE$ of all elliptic curves over $\Q$.
\emph{All these analytic results over $\Q$ are conditional on
the Riemann hypothesis for the $L$-functions of the elliptic curves
involved.}
At the time that some of these results were proved,
the assertion that the $L$-function admits an analytic continuation
to $\C$ was an assumption too, but today this is a consequence of 
the theorem that all elliptic curves over $\Q$ are modular~\cite{Breuil2001}.

D.~Goldfeld~\cite{Goldfeld1979} proved
the conditional bound $\overline{\Average}(r_{\an}) \le 3.25$
for the family of quadratic twists
of a fixed elliptic curve $E$ over $\Q$,
and conjectured that the correct constant was $1/2$.
The constant $3.25$ was later improved to $1.5$
by D.\,R.~Heath-Brown~\cite{Heath-Brown2004}*{Theorem~3}.

For the family of all elliptic curves over $\Q$,
A.~Brumer proved the conditional bound 
$\overline{\Average}(r_{\an}) \le 2.3$~\cite{Brumer1992}.
He also proved the same bound for elliptic curves 
over $\F_q(t)$ \emph{unconditionally}.
In the case of $\F_q(t)$, the inequality $r \le r_{\an}$ is known,
so one deduces $\overline{\Average}(r) \le 2.3$ in this setting.
Over $\Q$, the constant $2.3$ was improved 
to $2$ by Heath-Brown~\cite{Heath-Brown2004}*{Theorem~1}
and to $25/14$ by M.~Young~\cite{Young2006}.
The latter implied the (conditional) positivity 
of $\underline{\Prob}(r_{\an} \le 1)$, 
and then also of $\underline{\Prob}(r= r_{\an} \le 1)$,
because $r_{\an} \le 1$ implies $r=r_{\an}$ 
(\cites{Kolyvagin1988,Kolyvagin1990, Gross-Zagier1986}
with \cite{Bump-Friedberg-Hoffstein1990} or \cite{Murty-Murty1991}).

Conditional bounds on $\overline{\Average}(r_{\an})$ 
for other algebraic families
of elliptic curves and abelian varieties have been given
by \'E.~Fouvry and J.~Pomyka{\l}a~\cite{Fouvry-Pomykala1993},
P.~Michel~\cites{Michel1995,Michel1997}, J.~Silverman~\cite{Silverman1998}, 
and R.~Wazir~\cite{Wazir2004}.

\subsection{Distribution of Selmer groups}

For the family of elliptic curves $y^2=x^3+k$ over $\Q$,
\'E.~Fouvry proved $\overline{\Average}(3^{r/2}) < \infty$, 
by bounding the average size of the Selmer group associated
to a $3$-isogeny (a slight generalization of the Selmer groups
we have considered so far)~\cite{Fouvry1993}.
This implies that $\overline{\Average}(r) < \infty$ in this family.

Recall our notation $s_p'\colonequals \dim \Sel_p(E) - \dim E(\Q)[p]$.
For the family of quadratic twists of $y^2=x^3-x$ over $\Q$,
Heath-Brown proved not only that $\Average(s_2')=3$ but also that
\begin{equation}
\label{E:Selmer distribution}
	\Prob\left(s_2' = d \right)
	= \left( \prod_{j \ge 0} (1+2^{-j})^{-1} \right) 
	\left( \prod_{j=1}^d \frac{2}{2^j-1} \right)
\end{equation}
for each $d \in \Z_{\ge 0}$~\cites{Heath-Brown1993,Heath-Brown1994}.
P.~Swinnerton-Dyer \cite{Swinnerton-Dyer2008}
and D.~Kane \cite{Kane2013}
generalized this by obtaining the same distribution
for the family of quadratic twists of
any $E$ over $\Q$ with $E[2] \subset E(\Q)$
but no rational cyclic $4$-isogeny.
Heath-Brown's approach was used also by G.~Yu~\cite{Yu2006}
to prove finiteness of $\overline{\Average}(\#\Sel_2)$ 
for the family of all elliptic curves with $E[2] \subset E(\Q)$.
In certain subfamilies of this, surprises occur: see~\cite{Yu2005}.

A probabilistic model
predicting the distribution of $s_p'$ for any prime $p$
was proposed in~\cite{Poonen-Rains2012-selmer};
for $p=2$ the prediction is consistent with~\eqref{E:Selmer distribution}.

Earlier, C.~Delaunay~\cite{Delaunay2001} 
proposed a heuristic for the distribution of $\#\Sha$,
in analogy with the Cohen--Lenstra heuristics~\cite{Cohen-Lenstra1984}.

Finally, there is a conjecture 
that elliptic curves tend to have the smallest rank compatible
with the root number, which is expected to be equidistributed.
This was proposed in~\cite{Goldfeld1979} for the case of 
quadratic twists of a fixed curve, but it is probably true more generally.
In other words, it is expected that $\Prob(r=0)$ and $\Prob(r=1)$
are both $1/2$.
See also \cite{Katz-Sarnak1999b}*{\S5} 
and~\cite{Katz-Sarnak1999a}.

The three predictions above are compatible with
the equation $s_p' = \dim \Sha[p] + r$ arising from~\eqref{E:Selmer-Sha}: 
see~\cite{Poonen-Rains2012-selmer}*{\S5}.

\subsection{Average size of Selmer groups over function fields}

The closest parallel to the work of Bhargava and Shankar
is a 2002 article by A.\,J.~de Jong 
proving the analogue of Theorem~\ref{T:Sel3} for function fields, 
with a slightly weaker bound~\cite{DeJong2002}.
Namely, for any finite field $\F_q$ of characteristic not~$3$,
de Jong proved $\overline{\Average}(\#\Sel_3) \le 4 + \epsilon(q)$
for the family of all elliptic curves over $\F_q(t)$,
where $\epsilon(q)$ is an explicit rational function of $q$
tending to $0$ as $q \to \infty$.
This implies a corresponding bound for $\overline{\Average}(r)$
for such $\F_q(t)$.
Moreover, de Jong gave heuristics that in hindsight
hint that $\Average(\#\Sel_3) = 4$ 
not only for $\F_q(t)$ but also for $\Q$.

The approaches of de Jong and Bhargava--Shankar are similar.
Namely, both count integral models of geometric objects 
representing elements of $\Sel_n(E)$.
(For $n=3$, these objects are plane cubic curves.)
But the more delicate estimates, essential for obtaining
an asymptotically \emph{sharp} upper bound on 
$\sum_{E \in \EE_{<X}} \#\Sel_{\textup{$2$ or $3$}}(E)$
and a matching lower bound (for $\Sel_2$),
are unique to Bhargava--Shankar.

\section{Background: $\mathbf{n}$-diagrams and binary quartic forms}
\label{S:background}

Each element of $\HH^1(\Q,E[n])$ has a geometric avatar, 
called an $n$-diagram.
To count Selmer group elements,
we will count the possibilities for the coefficients 
of the polynomial equations defining their avatars.
We follow \cite{Cremona-et-al2008}*{\S1.3} in 
Sections~\ref{S:diagrams}--\ref{S:locally solvable}, 
and \cite{Birch-Swinnerton-Dyer1963} in Section~\ref{S:2-diagrams}.
The goal of this section is~\eqref{E:quartics and Selmer}.

\subsection{Diagrams}
\label{S:diagrams}

Fix a field $k$, a separable closure $\kbar$,
and an elliptic curve $E$ over $k$.
A \defin{diagram} for $E$ is a morphism of varieties 
from an $E$-torsor $C$ to a variety $S$.
An \defin{isomorphism of diagrams} 
is given by an isomorphism of $E$-torsors $C \to C'$
and an isomorphism of varieties $S \to S'$ making the obvious square
commute.

\subsection{$n$-diagrams}
\label{S:n-diagrams}

Let $O \in E(k)$ be the identity.
Fix an integer $n \ge 2$ with $\Char k \nmid n$.
The \defin{trivial $n$-diagram} 
is the diagram $E \to \PP^{n-1}$ determined by the linear system $|nO|$,
where $E$ is viewed as trivial $E$-torsor.
More generally, an \defin{$n$-diagram} 
is a twist $C \to S$ of the trivial $n$-diagram,
i.e., a diagram that becomes isomorphic to the trivial $n$-diagram
after base extension of both to $\kbar$.
In particular, $S$ must be a \defin{Brauer-Severi variety},
a twist of projective space.
(For this reason, $n$-diagrams are called Brauer-Severi diagrams
in~\cite{Cremona-et-al2008}*{\S1.3}.)

The automorphism group of the trivial $n$-diagram over $\kbar$
is given by $E[n]$ acting as translations on $E$
and acting compatibly on $\PP^{n-1}$.
Galois descent theory then yields a bijection
\begin{equation}
\label{E:n-diagrams}
	\frac{\{ \textup{$n$-diagrams for $E$} \}}
	{\textup{\scriptsize isomorphism} }
	\longleftrightarrow   
	\HH^1(k,E[n]).
\end{equation}

\begin{rema}
Elements of $\HH^1(k,E[n])$ are in bijection also with
geometric objects called $n$-coverings \cite{Cremona-et-al2008}*{\S1.2}.
But it is the $n$-diagrams that are easiest to count.
\end{rema}

\begin{rema}
\label{R:E[n] in PGL_n}
The action of $E[n]$ on $\PP^{n-1}$ 
is given by an injective homomorphism $E[n] \to \PGL_n$.
\end{rema}

\subsection{Solvable and locally solvable $n$-diagrams}
\label{S:locally solvable}

The homomorphism $\HH^1(k,E[n]) \to \HH^1(k,E)$ 
corresponds to sending an $n$-diagram \hbox{$C \to S$} to the torsor $C$.
Its kernel, which is isomorphic to $E(k)/nE(k)$ (cf.~\eqref{E:descent}),
corresponds to the set of $n$-diagrams $C \to S$ 
for which $C$ has a $k$-point;
such $n$-diagrams are called \defin{solvable}:
\begin{equation}
\label{E:solvable n-diagrams}
	\frac{\{ \textup{solvable $n$-diagrams for $E$} \}}
	{\textup{\scriptsize isomorphism} }
	\longleftrightarrow   
	\frac{E(k)}{nE(k)}.
\end{equation}

An $n$-diagram $C \to S$ over $\Q$ is \defin{locally solvable} 
if $C$ has a $\Q_p$-point for all $p \le \infty$.
In this case, $S$ too has a $\Q_p$-point for all $p \le \infty$,
and hence $S \isom \PP^{n-1}$, by the local-global principle for
the Brauer group.
(Not every $n$-diagram with $S \isom \PP^{n-1}$ is locally solvable, however.)
By~\eqref{E:descent}, we have a bijection
\begin{equation}
\label{E:locally solvable n-diagrams}
	\frac{\{ \textup{locally solvable $n$-diagrams for $E$} \}}
	{\textup{\scriptsize isomorphism} }
	\longleftrightarrow   
	\Sel_n(E).
\end{equation}

\subsection{Binary quartic forms}
\label{S:quartic forms}

With an eye towards Section~\ref{S:2-diagrams},
we consider the space $\Q[x,y]_4$ of binary quartic forms.
There is a left action of $\GL_2(\Q)$ on $\Q[x,y]_4$ given by 
$(\gamma \cdot f)(x,y) \colonequals f((x,y) \gamma)$ (we view $(x,y)$ as a row vector).
This induces an action of $\GL_2(\Q)$ on the algebra $\Q[a,b,c,d,e]$ 
of polynomial functions in the coefficients of
$f \colonequals a x^4 + b x^3 y + c x^2 y^2 + d x y^3 + e y^4$.
The subalgebra of $\SL_2$-invariants is 
$\Q[a,b,c,d,e]^{\SL_2} = \Q[I,J] = \Q[A,B]$, where
\begin{alignat*}{2}
  I &\colonequals 12ae - 3bd + c^2 &\qquad A &\colonequals -I/3\\
  J &\colonequals 72ace + 9bcd - 27ad^2 - 27eb^2 - 2c^3 &\qquad B &\colonequals -J/27.
\end{alignat*}
(Why $-1/3$ and $-1/27$?  To make the Jacobian statement 
in Section~\ref{S:2-diagrams} true.)
A quartic form is separable if and only if 
its discriminant $\Delta\colonequals -(4A^3+27B^2)$ is nonzero.
If $\gamma \in \GL_2(\Q)$ and $f \in \Q[x,y]_4$,
then $A(\gamma \cdot f) = (\det \gamma)^4 A(f)$
and $B(\gamma \cdot f) = (\det \gamma)^6 B(f)$.
Thus the twisted $\GL_2$-action 
$\gamma * f\colonequals (\det \gamma)^{-2} (\gamma \cdot f)$
induces a $\PGL_2$-action preserving $A$ and $B$.

\subsection{Locally solvable $2$-diagrams and binary quartic forms}
\label{S:2-diagrams}

Let $f \in \Q[x,y]_4$ be such that $\Delta(f) \ne 0$.
Let $C$ be the curve $z^2=f(x,y)$ 
in the weighted projective plane $\PP(1,1,2)$.
Let $\Jac C$ be its Jacobian.
Then there is an isomorphism between $\Jac C$
and the elliptic curve $y^2 = x^3 + A(f) x + B(f)$, 
and the isomorphism depends algebraically on the coefficients of $f$.
In fact, this makes $C \stackrel{(x:y)}\To \PP^1$ a $2$-diagram.
There is an approximate converse: any locally solvable $2$-diagram for $E$ 
is a degree $2$ morphism $C \to \PP^1$ 
equipped with an isomorphism $\Jac C \to E$
such that $C$ has a $\Q_p$-point for all $p \le \infty$;
such a curve $C$ is given by $z^2=f(x,y)$ in $\PP(1,1,2)$,
for some $f(x,y) \in \Q[x,y]_4$ such that $\Delta(f) \ne 0$.

Two locally solvable $2$-diagrams are isomorphic if and only if
the associated morphisms $C \to \PP^1$ and $C' \to \PP^1$ are isomorphic
\emph{forgetting the torsor structures}
(if there exists an isomorphism of such morphisms, 
there is also an isomorphism respecting the torsor structures,
because the automorphism group of $E$ over $\Q$ 
is never larger than $\{\pm1\}$).
And two such morphisms are isomorphic if and only if
corresponding quartic forms are $\PGL_2(\Q)$-equivalent
after multiplying one by an element of $\Q^{\times 2}$.
If the quartic forms already have the same invariants,
then the element of $\Q^{\times 2}$ is unnecessary.

Let $V=\Spec \Q[a,b,c,d,e]$ be the moduli space of quartic forms,
and let $\calS = \Spec \Q[A,B]$.
Let $V' \subset V$ and $\calS' \subset \calS$
be the open subvarieties defined by $\Delta \ne 0$.
There is a morphism $V \to \calS$ taking a quartic form to
its pair of invariants.
Let $V_{AB}$ be the fiber above $(A,B) \in \calS'(\Q)$.
Define the locally solvable subset $V(\Q)^{\ls} \subset V'(\Q)$ 
as the set of $f \in \Q[x,y]_4$ with $\Delta \ne 0$
such that $z^2=f(x,y)$ has a $\Q_p$-point for all $p \le \infty$.
Define $V_{AB}(\Q)^{\ls}$ similarly.
For each $(A,B) \in \calS'(\Q)$,
combining the previous paragraph 
with~\eqref{E:locally solvable n-diagrams} for $n=2$
and $E=E_{AB}$ yields a bijection
\begin{equation}
\label{E:quartics and Selmer}
	\leftquot{\PGL_2(\Q)}{V_{AB}(\Q)^{\ls}}
	\longleftrightarrow
	\Sel_2(E_{AB}).
\end{equation}
Which quartic forms on the left side 
correspond to the identity in $\Sel_2(E_{AB})$?
Those with a linear factor over $\Q$.

\begin{rema}
Similarly, $\Sel_3(E_{AB})$ can be related to $\PGL_3(\Q)$-orbits 
of ternary cubic forms; this is what is used to prove Theorem~\ref{T:Sel3}.
\end{rema}

\section{Proof of the theorem on 2-Selmer groups}

Our goal is to sketch the proof of Theorem~\ref{T:Sel2}.
For lack of space, our presentation will necessarily omit many details,
so the actual proof is more difficult than we might make it seem.
Also, in contrast to~\cite{Bhargava-Shankar-2selmer},
we will phrase the proof in adelic terms.
This has advantages and disadvantages as far as exposition is concerned, 
but does not really change any of the key arguments.

\subsection{Strategy of the proof}

For $X \in \R$, define 
\begin{equation}
\label{E:definition of S<X}
\begin{split}
	\calS(\Z)_{<X} &\colonequals \left\{\, (A,B) \in \Z^2 : \textup{$\Delta \ne 0$, $(A,B)$ is minimal and $H(A,B)<X$} \,\right\} \\
	V(\Q)^{\ls}_{<X} &\colonequals \textup{the subset of $V(\Q)^{\ls}$ mapping into $\calS(\Z)_{<X}$}.
\end{split}
\end{equation}
Summing the sizes of the sets in~\eqref{E:quartics and Selmer} 
over $(A,B) \in \calS(\Z)_{<X}$
yields
\begin{equation}
\label{E:sum of Selmer}
	\# \left( \leftquot{\PGL_2(\Q)}{V(\Q)^{\ls}_{<X}} \right) 
	= \sum_{E \in \EE_{<X}} \#\Sel_2(E).
\end{equation}
{}From now on, we forget about Selmer groups and estimate the left side
of~\eqref{E:sum of Selmer}.

If our job were to estimate the number of 
integral points in a region $\Omega \subset \R^n$,
we would compute the volume of $\Omega$
and argue that it is a good estimate provided that 
the shape of $\Omega$ is reasonable.
But according to~\eqref{E:sum of Selmer},
we need to count (orbits of) \emph{rational} points.
So instead of viewing $\Z$ as a lattice in $\R$,
we view $\Q$ as a lattice in the ring of adeles 
\[
	\Adeles\colonequals \left\{\, (x_p) \in \prod_{p \le \infty} \Q_p : 
	x_p \in \Z_p \textup{ for all but finitely many $p$} \,\right\}.
\]

How do we define an adelic region $V(\Adeles)^{\ls}_{<X}$
whose set of rational points is $V(\Q)^{\ls}_{<X}$?
Inspired by~\eqref{E:definition of S<X}, we define
\begin{align*}
	\calS(\R)_{<X} &\colonequals \left\{\, (A,B) \in \calS'(\R): H(A,B)<X \,\right\} \\
	\calS(\Z_p)^{\min} &\colonequals \left( (\Z_p \times \Z_p) - (p^4\Z_p \times p^6 \Z_p) \right) - \textup{(zeros of $\Delta$)}\\
	\calS(\Adeles)_{<X} &\colonequals \calS(\R)_{<X} \times \prod_{\textup{finite $p$}} \calS(\Z_p)^{\min} \\
	V(\Q_p)^{\ls} &\colonequals \left\{\, f \in V(\Q_p) : \textup{$\Delta \ne 0$ and $z^2=f(x,y)$ has a $\Q_p$-point} \,\right\} \\
	V(\Adeles)^{\ls} &\colonequals V(\Adeles) \intersect \prod_{p\le \infty} V(\Q_p)^{\ls} \\
	V(\Adeles)^{\ls}_{<X} &\colonequals 
	\textup{the subset of $V(\Adeles)^{\ls}$ mapping into $\calS(\Adeles)_{<X}$.}
\end{align*}
One might expect the rest of the proof to proceed as follows:
\begin{enumerate}[\upshape 1.]
\item Define an adelic measure on 
$\leftquot{\PGL_2(\Q)}{V(\Adeles)^{\ls}_{<X}}$
and compute its volume.
\item Show that $\# \left( \leftquot{\PGL_2(\Q)}{V(\Q)^{\ls}_{<X}} \right)$
is well approximated by that adelic volume.
\end{enumerate}
But statement~2 turns out to be false!
It will be salvaged by excluding the quartic forms with a linear factor,
i.e., those corresponding to the identity in a Selmer group.
In other words, it is $\sum_{E \in \EE_{<X}} \left(\#\Sel_2(E)-1 \right)$
that is approximated by the adelic volume.

\subsection{Computing the adelic volume}

The space $\Q_p$ has the usual Haar measure $\mu_p$
(Lebesgue measure if $p=\infty$).
The adelic measure on $\Adeles$ is the product of these.
The product $\prod_{\textup{finite $p$}} \mu_p(\calS(\Z_p)^{\min})$ converges,
so $\calS(\Adeles)_{<X}$ inherits an adelic measure from $\Adeles^2$.
In fact, $\mu_\infty(\calS(\R)_{<X}) = 4 \cdot 4^{-1/3} 27^{-1/2} X^{5/6}$
(area of a rectangle)
and $\mu_p(\calS(\Z_p)^{\min})=1-p^{-4}p^{-6}$,
and the product is $\mu(\calS(\Adeles)_{<X}) = c X^{5/6}$,
where $c\colonequals 2^{4/3} 3^{-3/2} \zeta(10)^{-1}$.

Although $V(\Adeles)^{\ls} \subset \Adeles^5$ 
is a restricted direct product and not a direct product, 
it is a union of direct products
that can be given an adelic measure as above.
The action of $\PGL_2(\Q)$ is measure-preserving,
so the quotient $\leftquot{\PGL_2(\Q)}{V(\Adeles)^{\ls}_{<X}}$
inherits the measure.

Let $\calE \to \calS'$ 
be the universal elliptic curve in short Weierstrass form.
Let $W$ be the moduli space of pairs $(f,P)$
where $f$ is a quartic form of nonzero discriminant
and $P$ is a point on $z^2=f(x,y)$.
The group scheme $\PGL_{2,\calS'}$ acts on $W$ (transforming both $f$ and $P$).
The forgetful $\calS'$-morphism $F \colon W \to V'$ is an $\calE$-torsor,
and is $\PGL_{2,\calS'}$-equivariant.
In fact, $W \to \calS'$ is a homogeneous space under 
$\calE \underset{\calS'}\times \PGL_{2,\calS'}$.
The stabilizer of $(x^3 y + A x y^3 + B y^4,(1:0:0))$ 
is $\calE[2]$ embedded diagonally (see Remark~\ref{R:E[n] in PGL_n}), so
$W \isom 
    \rightquot{\left( \calE \underset{\calS'}\times \PGL_{2,\calS'} \right)}
								{\calE[2]}$.
The quotient $\calS'$-morphism
$q \colon W \to \PGL_{2,\calS'}\backslash W \isom \calE/\calE[2] \isom \calE$
is a $\PGL_2$-torsor, and it turns out to admit a rational section.

We obtain a commutative (but not cartesian) diagram
\[
\xymatrix{
W \ar[r]^-q \ar[d]_-F & \calE \ar[d] \\
V' \ar[r] & \calS'.
}
\]
Let $W(\Adeles)_{<X}$ 
be the subset of $W\left(\prod_{p \le \infty} \Q_p \right)$
(not of $W(\Adeles)$!)
mapping into $V(\Adeles)^{\ls}_{<X}$.
Let $\calE(\Adeles)_{<X}$
be the subset of $\calE\left(\prod_{p \le \infty} \Q_p \right)$
mapping into $\calS(\Adeles)_{<X}$.

We define the measure of a subset of $W(\Adeles)_{<X}$
by integrating over $V(\Adeles)^{\ls}_{<X}$ the measure of the fibers of $F$,
where each full fiber, an $E_{AB}(\Adeles)$-torsor,
is assigned the mass $1$ Haar measure.
Define a measure on $\calE(\Adeles)_{<X}$ in the same way by integrating
over $\calS(\Adeles)_{<X}$.
It turns out that the fibers of 
$W(\Adeles)_{<X} \stackrel{q}\to \calE(\Adeles)_{<X}$ 
outside a measure-zero subset are $\PGL_2(\Adeles)$-torsors,
and that the Tamagawa measure $\mu_{\Tam}$ on these torsors 
is compatible with the measures on $W(\Adeles)_{<X}$ and $\calE(\Adeles)_{<X}$.

Now consider
\[
\xymatrix{
\leftquot{\PGL_2(\Q)}{W(\Adeles)_{<X}} \ar[r]^-q \ar[d]_-F & \calE(\Adeles)_{<X} \ar[d] \\
\leftquot{\PGL_2(\Q)}{V(\Adeles)^{\ls}_{<X}} \ar[r] & \calS(\Adeles)_{<X}.
}
\]
Working counterclockwise from $\calS(\Adeles)_{<X}$, we have
\begin{align*}
  \mu(\calE(\Adeles)_{<X}) &= \mu(\calS(\Adeles)_{<X}) \\
  \mu\left(\leftquot{\PGL_2(\Q)}{W(\Adeles)_{<X}} \right) &= \mu_{\Tam}\left(\rightquot{\PGL_2(\Adeles)}{\PGL_2(\Q)} \right) \;\; \mu(\calE(\Adeles)_{<X}) \\
  \mu\left(\leftquot{\PGL_2(\Q)}{W(\Adeles)_{<X}} \right) &= \mu\left(\leftquot{\PGL_2(\Q)}{V(\Adeles)^{\ls}_{<X}} \right),
\end{align*}
and the \defin{Tamagawa number} 
\[
	\tau(\PGL_2) \colonequals \mu_{\Tam}\left( \rightquot{\PGL_2(\Adeles)}{\PGL_2(\Q)} \right)
\]
is known to be $2$, so 
\begin{equation}
\label{E:adelic measure}
   \mu\left(\leftquot{\PGL_2(\Q)}{V(\Adeles)^{\ls}_{<X}} \right)
	= 2 \; \mu(\calS(\Adeles)_{<X}) = 2c X^{5/6}.
\end{equation}

\subsection{Counting rational points in adelic regions}

\begin{prop}[Denominator for $\Average(\#\Sel_2-1)$]
\label{P:number of elliptic curves}
As $X \to \infty$,
\[
	\sum_{E \in \EE_{<X}} 1 
	= \#\calS(\Z)_{<X} 
	= (1+o(1)) \; \mu\left( \calS(\Adeles)_{<X} \right).
\]
\end{prop}

\begin{proof}
The first equality is trivial.
Now, $\#\calS(\Z)_{<X}$
is the number of integral points $(A,B)$ in a large rectangle
that remain after sieving out those satisfying $p^4|A$ and $p^6|B$
for some prime $p$ and discarding those with $4A^3+27B^2=0$.
The sieving is elementary, and can be handled either by a M\"obius inversion 
argument~\cite{Brumer1992}*{Lemma~4.3},
or by sieving at the first few primes with the Chinese remainder theorem
and then arguing that the number of points removed by sieving
at all the remaining large primes is negligible.
This leaves $(1+o(1)) c X^{5/6}$ points.
The $(A,B)$ with $4A^3+27B^2=0$ have the form $(-3n^2,2n^3)$;
there are only $O(X^{1/6})$ such points of height up to $X$,
so discarding them does not affect the asymptotics.
(For related calculations over number fields, see~\cite{Bekyel2004}.)
\end{proof}

Let $V(\Q)^{\ls,\nl}_{<X}$ be the set of $f \in V(\Q)^{\ls}_{<X}$
that have no rational linear factor.
Most of the rest of the section will be devoted to the proof of
the following:

\begin{prop}[Numerator for $\Average(\#\Sel_2-1)$]
\label{P:number of Selmer elements}
As $X \to \infty$,
\begin{align*}
	\sum_{E \in \EE_{<X}} \left( \#\Sel_2(E) - 1 \right)
	&= \# \left( \leftquot{\PGL_2(\Q)}{V(\Q)^{\ls,\nl}_{<X}} \right) \\
	&= (1+o(1)) \; 
	\mu\left(\leftquot{\PGL_2(\Q)}{V(\Adeles)^{\ls}_{<X}} \right).
\end{align*}
\end{prop}

Ideally, we could choose a fundamental domain $\calF$ for
the action of $\PGL_2(\Q)$ on $V(\Adeles)^{\ls}_{<X}$
and simply count the rational points of $V(\Q)^{\ls,\nl}_{<X}$ in it.
In an attempt to construct such an $\calF$
we use the theory of integral models of binary quartic forms.

\begin{lemm}[Existence of integral 
models~\cite{Birch-Swinnerton-Dyer1963}*{Lemmas 3, 4, and~5}]
\label{L:integral}
Any locally solvable quartic form $f \in \Q[x,y]$ 
with $A \in 2^4 \Z$ and $B \in 2^6 \Z$
is $\PGL_2(\Q)$-equivalent to a quartic form in $\Z[x,y]$.
\end{lemm}

To avoid some inconsequential technicalities,
we ignore the $2^4$ and $2^6$ in the rest of our exposition.
We also ignore the points in 
$V(\Q)^{\ls,\nl}_{<X}$ with a nontrivial stabilizer in $\PGL_2(\Q)$:
one can show that the contribution from these is negligible.

Define $\Zhat \colonequals \prod_{\textup{finite $p$}} \Z_p$,
and define $V(\Zhat)^{\ls}$ in the obvious way.
The proof of Lemma~\ref{L:integral}
shows also that every quartic form in $V(\Adeles)^{\ls}_{<X}$
is $\PGL_2(\Q)$-equivalent to one in $V(\R)^{\ls}_{<X} \times V(\Zhat)^{\ls}$
(if we ignore $2^4$ and $2^6$).

An explicit fundamental domain 
$\calF^\R$ for $\leftquot{\PGL_2(\Z)}{V(\R)^{\ls}}$
can be obtained by combining
Gauss's fundamental domain for $\leftquot{\PGL_2(\Z)}{\PGL_2(\R)}$
with an easily described
fundamental domain for $\leftquot{\PGL_2(\R)}{V(\R)^{\ls}}$.
By the previous paragraph, 
if $\calF^\R_{<X}\colonequals \{\mbox{$f \in \calF^\R$}: H(f) < X\}$,
then every quartic form in $V(\Adeles)^{\ls}_{<X}$
is $\PGL_2(\Q)$-equivalent to one in the subset 
$\calF \colonequals \calF^\R_{<X} \times V(\Zhat)^{\ls}$ (if we ignore $2^4$ and $2^6$).
Rational points in $\calF$ now are integral points in $\calF^\R_{<X}$
satisfying local solvability.

There are two problems with $\calF$:
\begin{enumerate}
\item 
The region $\calF^{\R}_{<X}$ has a narrow cusp stretching to infinity,
which makes it hard to approximate its number of integral points
by its volume.
\item 
The set $\calF$ is not a fundamental domain!
(Although integral points in $\calF^{\R}_{<X}$ 
cannot be $\PGL_2(\Z)$-equivalent,
they can still be $\PGL_2(\Q)$-equivalent.
This phenomenon can happen only for quartic forms whose discriminant
is divisible by $p^2$ for some prime $p$.
For instance,
\[
	p^2 x^4 + p x^3 y + x^2 y^2 + x y^3 + y^4
	\quad\textup{and}\quad
	x^4 + x^3 y + x^2 y^2 + p x y^3 + p^2 y^4
\]
are $\PGL_2(\Q)$-equivalent.)
\end{enumerate}

Problem~1 is solved by an idea from~\cite{Bhargava2005}*{\S2.2}, 
namely to average over a ``compact continuum'' of $\PGL_2(\R)$-translates 
of $\calF^{\R}_{<X}$.
This fattens the cusp enough that the volume estimate applies
to the ``main body'' obtained by cutting off most of the cusp.
It turns out that the severed part contains a disproportionately 
large number of integral points, but they are all from
the quartic forms with a rational linear factor;
on the other hand, 
the main body contains few quartic forms with a rational linear factor;
this explains why we exclude them to obtain a count approximated
by a volume.

Problem~2 is more serious.
One solution might be to find some way to select one 
$\PGL_2(\Z)$-orbit of integral quartic forms 
within each $\PGL_2(\Q)$-equivalence class.
A more elegant solution is to select them all,
but to weight each one by $1/n$ where $n$ is the number of
possibilities.
By an argument involving the class number of $\PGL_2$ being $1$,
this weight turns out to be expressible as a product over all primes $p$
of local weights defined analogously in terms of the number
of $\PGL_2(\Z_p)$-orbits of quartic form over $\Z_p$ 
within a $\PGL_2(\Q_p)$-equivalence class.
(Strictly speaking, one also needs to take into account 
the orders of stabilizers in defining these weights.)
The situation is now similar to that in the proof
of Proposition~\ref{P:number of elliptic curves},
in which we counted integral points with a weight that was
$1$ or $0$ according to whether it was minimal (at every prime $p$) or not.
If we approximate the actual weights by the product of the local weights
at the first few primes, then the weighted count of integral points
can be approximated by a weighted volume.
It remains to show that the number of points at which
the actual weight differs from the approximate weight
is negligible.
The local weight turns out to be $1$ whenever $p^2 \nmid \Delta$,
so it suffices to sum the following bound over all primes $p$
beyond a large number:

\begin{lemm}[\cite{Bhargava-Shankar-2selmer}*{Proposition~3.16}]
\label{L:uniformity}
The number of $\PGL_2(\Z)$-orbits of integral quartic forms
of height less than $X$ such that $\Delta \ne 0$ and $p^2|\Delta$
is $O(p^{-3/2} X^{5/6})$.
\end{lemm}

The proof of Lemma~\ref{L:uniformity} is the trickiest part
of the whole argument.
The observation that $\Delta$ is a polynomial in $a,b,c,d,e$
is enough to prove Lemma~\ref{L:uniformity}
for primes $p$ up to a small fractional power of $X$,
but it is not known for an arbitrary polynomial 
how to obtain suitable bounds 
on the number of values divisible by the square of a 
\emph{larger} prime~\cites{Granville1998,Poonen2003-squarefree}.
Bhargava and Shankar resolve the difficulty in a surprising way:
using~\cite{Wood-thesis}*{Theorem~4.1.1}, 
they identify the set of quartic forms
with the set of quartic rings with monogenized 
cubic resolvent, 
which admits an (at most $12$)-to-$1$ map
to the much larger set of quartic rings with cubic resolvent,
which is in bijection with $(\GL_2(\Z) \times \SL_3(\Z))$-orbits
of pairs of ternary quadratic forms~\cite{Bhargava2004III}*{Theorem~1}.
Then they do the counting in this larger set, whose size
was calculated in~\cite{Bhargava2005}*{Theorem~7}.

This concludes the sketch of the proof 
of Proposition~\ref{P:number of Selmer elements}.

\begin{rema}
The role of Lemma~\ref{L:uniformity} is to ensure that we
are not \emph{overcounting} orbits.
Without Lemma~\ref{L:uniformity}, we could still deduce
$\overline{\Average}(\#\Sel_2) \le 3$.
\end{rema}

\begin{rema}
Calculations related to Lemma~\ref{L:uniformity} are used 
in~\cite{Bhargava-Shankar-2selmer} 
to compute not only $\Average(\#\Sel_2)$,
but also other averages,
such as the average size of the $2$-torsion subgroup of the 
class group of a maximal cubic order 
equipped with an element generating it as a ring.
\end{rema}

\subsection{End of proof}

Dividing Proposition~\ref{P:number of Selmer elements}
by Proposition~\ref{P:number of elliptic curves}
and using the volume relation~\eqref{E:adelic measure} yields
\[
	\Average(\#\Sel_2-1) = \tau(\PGL_2) = 2.
\]
Add $1$.

\section*{Acknowledgements} 

I thank Manjul Bhargava for explaining to me many details
of his work with Arul Shankar.
I thank also 
K\k{e}stutis \v{C}esnavi\v{c}ius, 
Jean-Louis Colliot-Th\'el\`ene, 
John Cremona, 
\'Etienne Fouvry, 
Benedict Gross,
Ruthi Hortsch, 
Jennifer Park, 
Joseph H. Silverman, 
Jack Thorne, 
and
Jeanine Van Order
for comments.


\end{document}